\documentclass[12pt]{article}
\usepackage{graphicx}
\usepackage{amssymb}
\usepackage{amsmath}
\usepackage{amsfonts}
\usepackage{amsthm}
\usepackage[english]{babel}
\usepackage[latin1,ansinew]{inputenc}
\usepackage{a4wide}
\usepackage{color}
\usepackage{float}

\newcommand{\be}{\begin{eqnarray}}
\newcommand{\ee}{\end{eqnarray}}
\newtheorem{theo}{Theorem}

\newtheorem{lemma}{Lemma}

\newcommand{\R}{\mathbb R}

\begin{document}

\title{Non-Existence and Existence of Shock Profiles\\ in the Bemfica-Disconzi-Noronha Model}
\author
{\it Heinrich Freist\"uhler\thanks{Department of Mathematics, University of Konstanz, 
78457 Konstanz, Germany. Supported by DFG Grant No.\ FR 822/10-1.}
}

\date{March 27, 2021}

\maketitle

\begin{abstract}
This note studies a four-field hyperbolic PDE model that was recently introduced by Bemfica, Disconzi, 
and Noronha for the pure radiation fluid with viscosity, and asks whether shock waves 
admit continuous profiles in this description. The model containing two free parameters $\mu,\nu$ 
and being causal whenever one chooses $(\mu,\nu)$ from a certain range $\mathcal C\subset\R^2$, 
this paper shows that for any choice of $(\mu,\nu)$ in the interior of $\mathcal C$,  
there is a dichotomy in so far as (i) shocks of sufficiently small 
amplitude admit profiles and (ii) certain other, thus necessarily non-small, shocks do not.
This finding does not preclude the possibility that if one chooses $(\mu,\nu)$ from a specific
part $\mathcal S$ of the boundary of $\mathcal C$, the dichotomy disappears and \emph{all} shocks 
have profiles; 
the parameter set $\mathcal S$ corresponds to the ``sharply causal'' case, in which one of the 
characteristic speeds of the dissipation operator is the speed of light.  
\end{abstract}
\newpage
\section{Introduction}
In their admirable 2018 paper \cite{BDN}, Bemfica, Disconzi, and Noronha have proposed 
a four-field PDE formulation 
\be\label{NS}
\partial_\beta(T^{\alpha\beta}+\Delta T^{\alpha\beta})=0
\ee
for the dynamics of the pure radiation fluid, ideally
\be
T^{\alpha\beta}=\frac{\partial (p(\theta)\psi^\alpha)}{\partial \psi_\beta}
=\theta^3p'(\theta)\psi^\alpha\psi^\beta+p(\theta)g^{\alpha\beta},
\quad
\psi^\gamma=\frac{u^\gamma}\theta,
\quad 
p(\theta)=\frac13\theta^4,
\ee
with a dissipation tensor
\be
\label{DeltaT}
\Delta T^{\alpha\beta}=
-B^{\alpha\beta\gamma\delta}(\psi)
\displaystyle{\frac{\partial\psi_\gamma}{\partial x^\delta}}
\ee
in which the classical Eckart viscosity ansatz \cite{E}
\be
\begin{aligned}\label{Bvisc}
B^{\alpha\beta\gamma\delta}_{visc}&=
\Pi^{\alpha\gamma}\Pi^{\beta\delta}+\Pi^{\alpha\delta}\Pi^{\beta\gamma}
-\frac23\Pi^{\alpha\beta}\Pi^{\gamma\delta}\quad\text{with}\quad
\Pi^{\alpha\gamma}=u^\alpha u^\beta+g^{\alpha\beta}
\end{aligned}
\ee
is augmented as\footnote{Our notation is slightly different from theirs.
We write $B_{ther},B_{velo}$ for the two characteristic pieces of the augmentation in order to 
make their relation to 
first-order equivalence transformations \cite{FT18} apparent, and use the thermodynamically 
natural Godunov variable $\psi^\gamma$ as in \cite{RS,FT14,F20}.} 
\be\label{Bis}
B^{\alpha\beta\gamma\delta}
=\eta
B^{\alpha\beta\gamma\delta}_{visc}-
\mu B^{\alpha\beta\gamma\delta}_{ther}-\nu B^{\alpha\beta\gamma\delta}_{velo}
\ee
with
\be\label{Bthervelo}
B^{\alpha\beta\gamma\delta}_{ther}=
(3 u^\alpha  u^\beta+\Pi^{\alpha\beta})
(3 u^\gamma u^\delta+\Pi^{\gamma\delta}),
\quad
B^{\alpha\beta\gamma\delta}_{velo}=
(u^\alpha {\Pi^\beta}_\epsilon+ u^\beta{\Pi^\alpha}_\epsilon) 
(u^\gamma\Pi^{\delta\epsilon}+ u^\delta\Pi^{\gamma\epsilon}).
\ee
Like that of a different augmentation introduced by B.\ Temple and the author in \cite{FT14,FT18},
the purpose of the combination \eqref{Bis} is to make the dissipation causal and thus remedy a 
well-known deficiency of the Eckart theory.\footnote{Recall Israel's work \cite{HL} for an early definitive 
statement of this deficiency and the first stringent attempt to handle it. Many others followed, cf.\ 
the introductions of \cite{FT14} and \cite{F20}.} 
It was shown in \cite{BDN} that 
this formulation, which we will here briefly refer to as `the BDN model', is indeed causal if and only if, 
relative to the classical coefficient $\eta$ of viscosity, the
coefficients $\mu$ and $\nu$ of the ``regulators'' 
$ 
B^{\alpha\beta\gamma\delta}_{ther},
B^{\alpha\beta\gamma\delta}_{velo}
$ 
satisfy
\be\label{generallycausal}
\mu\ge \frac43\eta\quad\text{and}\quad 
\nu\le\left(\frac1{3\eta}-\frac1{9\mu}\right)^{-1}.
\ee
The present note focusses on shock waves, whose ideal version is given by discontinuous  
solutions to the corresponding inviscid, Euler, equations
\be\label{E}
\partial_\beta T^{\alpha\beta}=0
\ee
of the (prototypical) form 
\be\label{lsw}
\psi(x)=\begin{cases}
       \psi_-,&x^\beta \xi_\beta<0,\\
       \psi_+,&x^\beta \xi_\beta>0,
       \end{cases}
\ee
and asks whether they can be properly represented in the viscous setting \eqref{NS}.
A standard way to achieve such representation of a `viscous shock wave' is a `dissipation profile', i.e., 
a regular solution of \eqref{NS} that depends also only on $x^\beta \xi_\beta$
and connects the two states forming the shock, in other words, a solution of the ODE  
\be\label{profeq}
\xi_\beta\xi_\delta B^{\alpha\beta\gamma\delta}(\psi)\psi_\gamma'
=
\xi_\beta T^{\alpha\beta}(\psi)-q^\alpha,
\quad 
q^\alpha:=\xi_\beta T^{\alpha\beta}(\psi_\pm),
\ee
on $\R$ which is heteroclinic to them,
\be
\label{hetero}
\hat\psi(-\infty)=\psi_-,\quad  
\hat\psi(+\infty)=\psi_+.
\ee
The technical main purpose of this note is to show the following.
\begin{theo}
For any choice of the coefficients $\eta,\mu,\nu>0$, \\ 
(i) every Lax shock 
of sufficiently small amplitude possesses a dissipation profile, and \\ 
(ii) if the coefficients satisfy the strict causality condition
\be\label{strictlycausal}
\mu\ge \frac43\eta\quad\text{and}\quad 
\nu< \left(\frac1{3\eta}-\frac1{9\mu}\right)^{-1},
\ee
then there always exist other shock waves that do not admit a dissipation profile.
\end{theo}
Besides the dichotomy that Theorem 1 states so distinctly, we draw the reader's attention to
the slight difference between the two causality conditions \eqref{generallycausal} and 
\eqref{strictlycausal}. To understand its meaning note that  
all characteristic speeds $\sigma$ 
of the operator 
$$
B^{\alpha\beta\gamma\delta}
\displaystyle{\frac{\partial^2\psi_\gamma}{\partial x^\beta\partial x^\delta}}
$$
satisfy $0\le \sigma^2\le 1$ if and only if both inequalities 
in \eqref{generallycausal} hold, and one of the speeds is luminal, $\sigma^2=1$, if and only 
if 
\be\label{sharplycausal}
\mu\ge \frac43\eta\quad\text{and}\quad 
\nu=\left(\frac1{3\eta}-\frac1{9\mu}\right)^{-1},
\ee
a condition that we therefore call \emph{sharp causality}.  

Obviously, Theorem 1 has the following two implications :
\goodbreak

(A) In the strictly causal case \eqref{strictlycausal}, upon variation from 
small to large amplitudes, necessarily some kind of transition occurs in the phase portrait 
of the dynamical system \eqref{profeq}.

(B) As it is well possible (though not formally proved in this paper, cf.\ Remark 1 at the end of 
Section 3)
that \emph{all} shocks have profiles when 
\be\label{nuisnu*}
\nu=\nu_*(\eta,\mu)=\left(\frac1{3\eta}-\frac1{9\mu}\right)^{-1},
\ee
it seems that for the modelling one might prefer this sharply causal tuning over others.

Both (A) -- cf.\ Figure 1 -- and (B) are under further investigation \cite{P}.

\vskip 1cm
\begin{figure}[H]
	\begin{center}
		\vspace{-1cm}
		\includegraphics[trim =60mm 60mm 40mm 50mm ,scale=0.57]{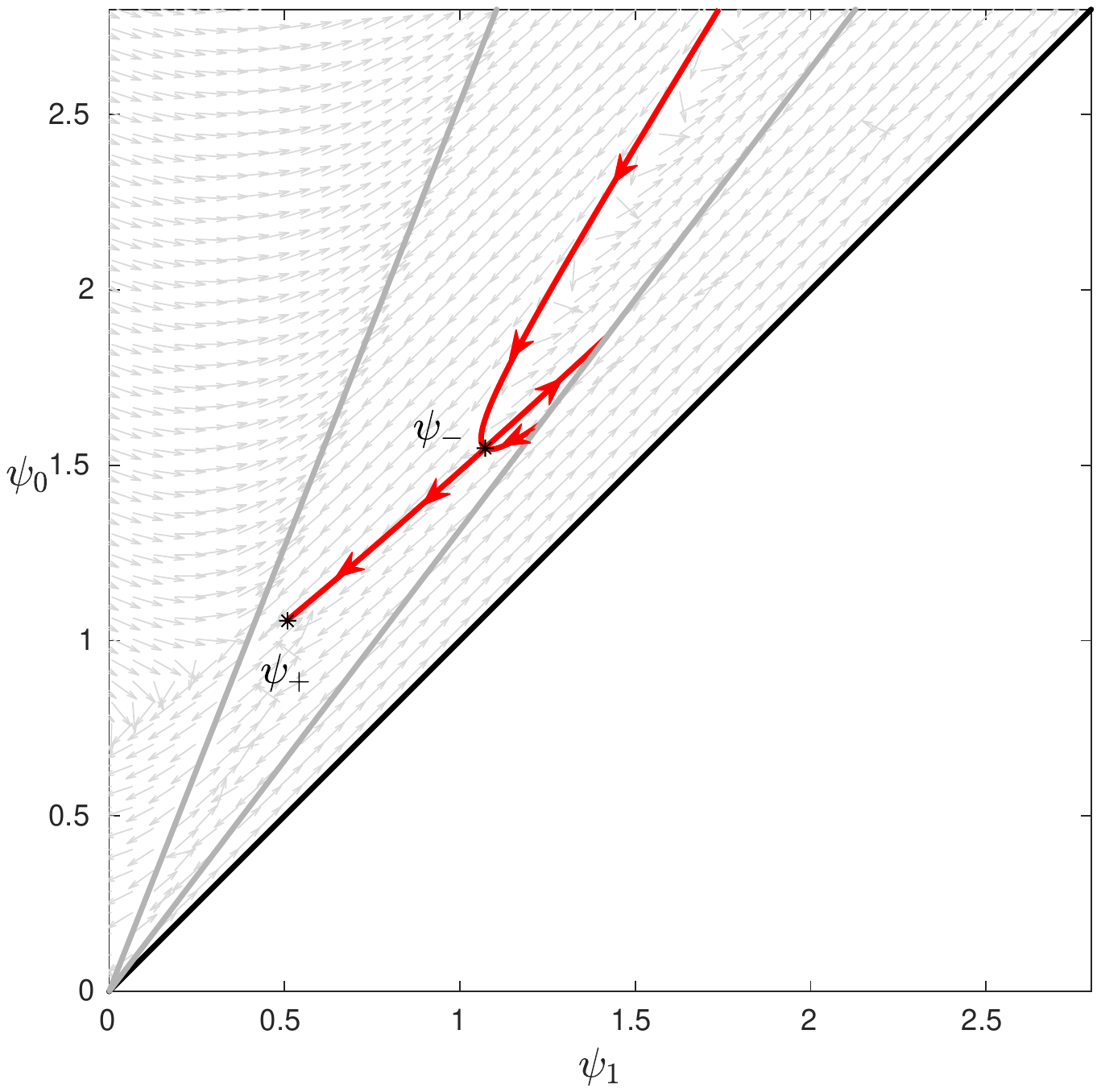}
		\quad
		\includegraphics[trim =30mm 60mm 60mm 56mm ,scale=0.57]{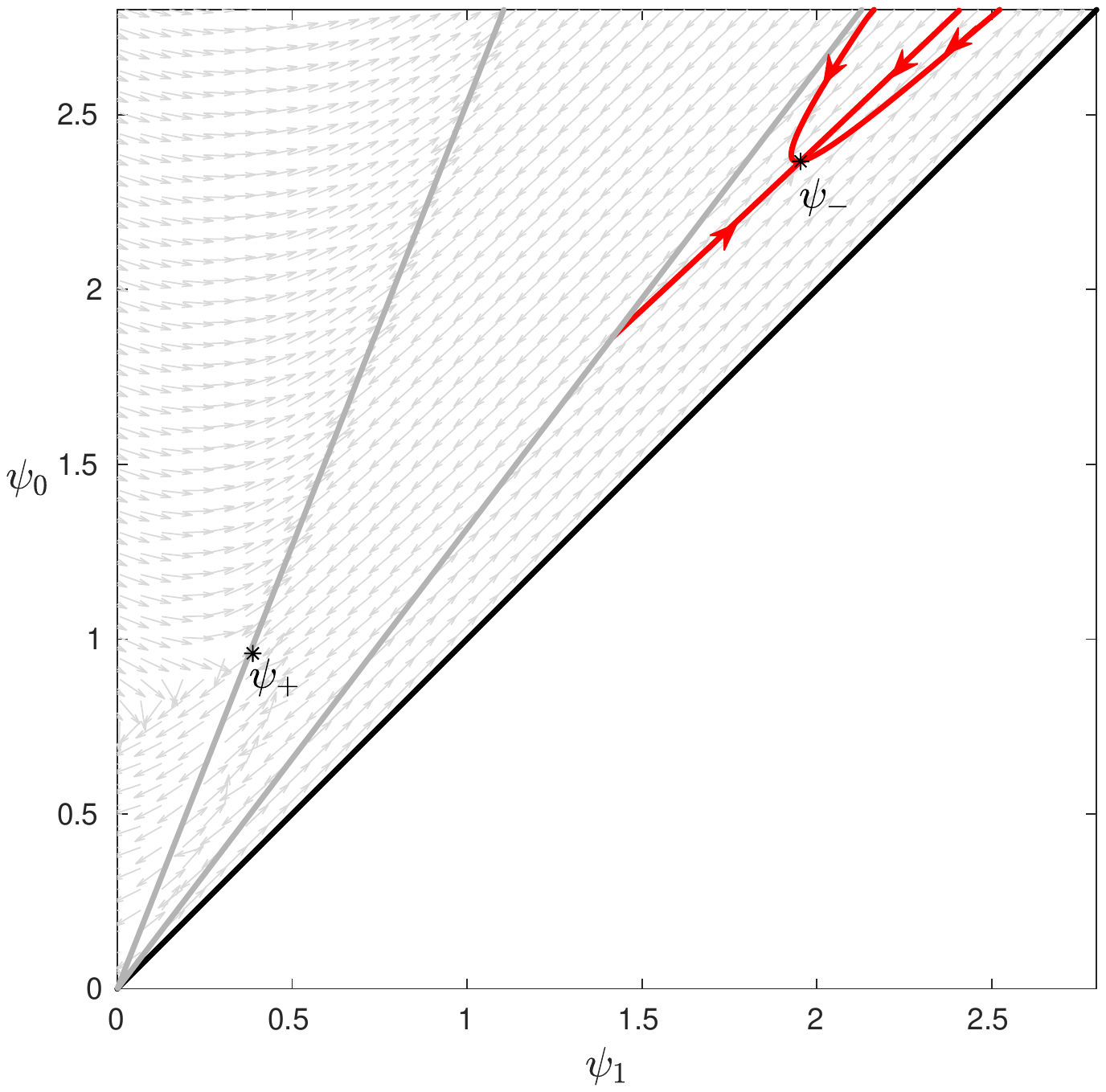}
  \caption{Phase portrait of \eqref{profeq} for strictly causal parameter values $(\eta,\mu,\nu)=(1,7,20)$. 
   Left: $\tilde q=31/40$\ . Right: $\tilde q=34/40$. 
Gray lines indicate det $B=0$. Plots by V. Pellhammer.}
	\end{center}
\end{figure} 

Dissipation profiles for fluid dynamical shock waves have been widely studied.
For a selection of aspects, including the interesting (though quite different) 
heteroclinic bifurcation in standard magnetohydrodynamics, the interested reader
is referred to \cite{We,Gi,MP,FT14,CS,FS,FP20}.

As regards the present paper, we parametrize the ideal shock waves of the pure radiation fluid  
in Section 2, Section 3 contains the proof of the theorem, and a brief appendix 
concisely reconsiders the causality conditions \eqref{generallycausal}, \eqref{sharplycausal}.

\section{Rankine-Hugoniot conditions of pure radiation}
\setcounter{equation}0
Due to Lorentz invariance, we can restrict attention to shocks of speed $s=0$, and because
of the system's natural isotropy, we may assume that the spatial direction of propagation 
is $(1,0,0)$. Correspondingly, we fix the spatiotemporal normal to the shock as 
$$(\xi_0,\xi_1,\xi_2,\xi_3)=(0,1,0,0),$$ 
henceforth consider only states $\psi$ with $\psi_2=\psi_3=0$, and 
let the indices run over 0 and 1 instead from 0 to 3; thus, for instance, 
the metric $g^{\alpha\gamma}$ and 
the projection $\Pi^{\alpha\gamma}=g^{\alpha\gamma}+u^\alpha u^\gamma$
on the orthogonal complement of the velocity $u^\alpha$  
are given by the matrices
$$
\left( g^{\alpha\gamma}\right)
=\begin{pmatrix}
-1&0\\0&1
\end{pmatrix},
\quad
\left(\Pi^{\alpha\gamma}\right)
=\begin{pmatrix}
 u_1^2&u^0 u^1\\
 u^0 u^1&u_0^2 
 \end{pmatrix}.
$$
On the state space $\Psi\equiv\{\psi=(\psi^0,\psi^1)\in\R^2:\psi^0>|\psi^1|\}$, we use 
the temperature and the velocity, 
$$
\theta=(-\psi_\gamma\psi^\gamma)^{-1/2},\quad
(u,v):=(u^0,u^1)=\theta(\psi^0,\psi^1)\ \text{with }u=(1+v^2)^{1/2},
$$ 
as coordinates. All possible shocks can be identified by screening the preimage set 
of the $\xi_\beta$ component of the ideal stress $T^{\alpha\beta}$ \cite{D,T}:
\begin{lemma}
On $\Psi$, the equation 
\be\label{Tisq}
T^{\alpha1}(\psi)=q^\alpha
\ee
has more than one solution if and only if 
$$
q^1>0
\quad\text{and}\quad
(q^1)^2<(q^0)^2<\frac2{\sqrt3}(q^1)^2.
$$ 
In that case, it has exactly two, and these two states form a standing Lax shock 
in right-moving or left-moving flow if $q^0>0$ or $q^0<0$, respectively.  
\end{lemma}
\begin{proof}
As the relation \eqref{Tisq}$_1$, $T^{11}(\psi)=q^1$, is equivalent to 
$$
\frac43v^2+\frac13=q^1\theta^{-4},
$$
its solution set is empty if $q^1\le 0$. If $q^1>0$, it is the curve
$$
H^1(q^1)=\{h^1(v,q^1):v\in\R\}
$$
with
$$
h^1(v,q^1)=\frac1{\theta^1(v,q^1)}\left(\sqrt{1+v^2},v\right),
\quad\text{where } 
\theta^1(v,q^1):=\left(\frac1{q^1}\left(\frac43v^2+\frac13\right)\right)^{-1/4}.
$$
As the relation \eqref{Tisq}$_0$, $T^{01}(\psi)=q^0$, is equivalent to 
$$
\left(\frac43\right)^2(1+v^2)v^2=(q^0)^2\theta^{-8}\quad\text{and}\quad q^0v\ge 0, 
$$
its solution set is $H^0(0)=(0,\infty)\times\{0\}$ if $q^0=0$, but otherwise it is the curve
$$
H^0(q^0)=\{h^0(v,q^0):q^0v>0\}
$$
with
$$
h^0(v,q^0)=\frac1{\theta^0(v,q^0)}\left(\sqrt{1+v^2},v\right),
\quad\text{where } 
\theta^0(v,q^0):=\left(\left(\frac43\right)^2\frac{(1+v^2)v^2}{(q^0)^2}\right)^{-1/8}.
$$
\goodbreak
Being the intersection of the two curves $H^0(q^0)$ and $H^1(q^1)$, 
the solution set $H(q)$ of \eqref{Tisq} can have more than one element only if we have 
$q^0\neq 0$. As \eqref {Tisq} is invariant under the reflection 
$q^0\mapsto -q^0, u^1\mapsto -u^1$, it is w.\ l.\ o.\ g.\ that we henceforth assume 
\be
q^1>0\quad\text{and}\quad q^0>0.
\ee
Observing that 
$$
\begin{aligned}
h^0(v,q^0)=h^1(v,q^1)&\Leftrightarrow \theta^0(v,q^0)=\theta^1(v,q^1)
\\
&\Leftrightarrow (4v^2+1)^2=16\tilde q^2v^2(1+v^2)\\
&\Leftrightarrow (16(1-\tilde q)v^4+8(1-2\tilde q)v^2+1
\end{aligned}
$$
with 
$$
\tilde q=(q^1/q^0)^2,
$$
we find that the equation $h^0(v,q^0)=h^1(v,q^1)$ has two different positive solutions $v_->v_+$
if and only if $$3/4<\tilde q<1$$ and in that case, these solutions are given through 
$$
v_\pm^2(\tilde q)=\frac{(2\tilde q-1)\mp\sqrt{\tilde q(4\tilde q-3)}}{4(1-\tilde q)}. 
$$
We note that 
$$
\lim_{\tilde q\searrow 3/4}v_\pm(\tilde q)=:v_*=\frac1{\sqrt2}
$$
while
$$
\lim_{\tilde q\nearrow 1}v_+(\tilde q)=\frac1{2\sqrt2}\quad\text{and}\quad
\lim_{\tilde q\nearrow 1}v_-(\tilde q)=\infty.
$$
I.\ e., the case $\tilde q\searrow 3/4$ corresponds to 
the zero amplitude and $\tilde q\nearrow 1$ to the 
infinite amplitude limit. 

As \eqref{Tisq} is also invariant under homotheties $(\psi,q) \mapsto(a\psi,a^{-2}q),a>0,$
it again causes no loss of generality when we from now on also assume that  
$
q^1=1. 
$ 
This means that we consider a $\tilde q$-dependent pair of states  
$$
(\psi_-(\tilde q),\psi_+(\tilde q))\quad\text{with }\psi_\pm=
\left(((4/3)v^2+1/3)^{1/4}((1+v^2)^{1/2},v)\right)|_{v=v_\pm(\tilde q)},
\quad
3/4<\tilde q<1,
$$
that originates from the bifurcation point 
$$ 
\psi_*=\frac1{\sqrt2}\left(\sqrt3,1\right)
$$
when $\tilde q$ grows starting from its lower limiting value $3/4$.
Since $v_->v_+>0$, $\psi_-$ as a left-hand (upstream) state and $\psi_+$ as a right-hand 
(downstream) state form a shock that stands in a right-moving decelerating flow.

To confirm that it is a Lax shock, we observe that up to a positive scalar factor, 
the Jacobian 
$$
\frac{\partial T^{\alpha1}}{\partial \psi_\gamma}
=
\theta^2p_\theta
(g^{\alpha1}u^\gamma +g^{\alpha\gamma}u^1 +g^{1\gamma}u^\alpha)
+(\theta^3p_\theta)_\theta u^\alpha u^1 u^\gamma
=\frac43\theta^5\left(
(g^{\alpha1}u^\gamma +g^{\alpha\gamma}u^1 +g^{1\gamma}u^\alpha)
+6 u^\alpha u^1 u^\gamma\right)
$$
is given by the matrix
\be\label{dT}
\begin{pmatrix}
 -v+6u^2v & u+6uv^2\\
 u+6uv^2 & 3v+6v^3 
\end{pmatrix}
=
\begin{pmatrix}
 v(6v^2+5)&\sqrt{1+v^2}(6v^2+1)\\
 \sqrt{1+v^2}(6v^2+1)&v(6v^2+3)
\end{pmatrix},
\ee
whose determinant 
\be\label{detA}
v^2(6v^2+5)(6v^2+3)-(1+v^2)(6v^2+1)^2=2v^2-1
\ee
is positive for $v^2>1/2$ and negative for $v<1/2$. Thus both characteristic speeds 
are positive at $\psi_-$, while for $\psi_+$ one of them is positive and 
the other negative: this defines a 1-shock \cite{Lx}. 
\end{proof}

\section{Shock profiles in the BDN picture}  

\setcounter{equation}0
We write the profile equation 
$$
B^{\alpha1\gamma1}(\psi)\psi_\gamma'=T^{\alpha 1}(\psi)-q^\alpha. 
$$
and 
a \emph{scaled} version of its linearization at a given state $\bar\psi$,
$$
B^{\alpha1\gamma1}(\bar\psi)\psi_\gamma'=
\frac{\partial T^{\alpha1}}{\partial \psi_\gamma}
(\bar\psi)\psi_\gamma,
$$
as
\be\label{profileODE1}
B(\psi)\psi'=F(\psi,\tilde q)
\ee
and 
\be\label{linprofileODE1}
B(\bar\psi)\psi'=A(\bar\psi)\psi, 
\ee
respectively, with
\be\label{B}
B
=\tilde\eta
\tilde B_{visc}-\mu B_{ther}-\nu B_{velo}
\ee
where
\be\label{Bviscpsi}
\tilde B_{visc}(\psi)
=
\begin{pmatrix}
u^2v^2&-u^3v\\-u^3v&u^4
\end{pmatrix},\qquad\tilde\eta=\frac43\eta,\qquad
\ee
\be\label{Bther}
B_{ther}(\psi)
=
\begin{pmatrix}
16u^2v^2&-4uv(4v^2+1)\\ 
-4uv(4v^2+1)&(4v^2+1)^2                                                                                         
\end{pmatrix},
\ee
\be\label{Bvelo}
\begin{aligned}
B_{velo}(\psi)
&=
\begin{pmatrix}
(u^2+v^2)^2 &-2(u^2+v^2)uv\\
-2(u^2+v^2)uv&4u^2v^2
\end{pmatrix},
\end{aligned}
\ee
\goodbreak
and
$$
A(\psi)
=\begin{pmatrix}
 v(6u^2-1)&-u(6v^2+1)\\
 -u(6v^2+1)&v(6v^2+3) 
 \end{pmatrix};
$$
according to \eqref{Bvisc},\eqref{Bthervelo}, and \eqref{dT}.  
(The minus signs in the offdiagonal entries of the above matrices are 
induced by the distinction between contra- and covariant indices, 
$\psi^\epsilon=g^{\epsilon\gamma}\psi_\gamma.$)
\goodbreak

For any $\tilde q\in (3/4,1)$, the two states $\psi_-(\tilde q),\psi_+(\tilde q)$ we have 
identified above, obviously are (the only) rest points of the ODE system \eqref{profileODE1}.

We first consider the small-amplitude case.
At the bifurcation point $\psi_*$, all four matrices 
$$
A(\psi_*)
=\frac1{\sqrt2}\begin{pmatrix}
 8&-4\sqrt{3}\\
 -4\sqrt{3}&6           
 \end{pmatrix},
\quad
\tilde B_{visc}(\psi_*)
=
\begin{pmatrix}
u^2v^2&-u^3v\\-u^3v&u^4
\end{pmatrix}
=
\frac34
\begin{pmatrix}
1&-\sqrt3 \\
-\sqrt3&3 
\end{pmatrix},
$$
$$
B_{ther}(\psi)
=
\begin{pmatrix}
16u^2v^2&-4uv(4v^2+1)\\ 
-4uv(4v^2+1)&(4v^2+1)^2                                                                                         
\end{pmatrix}
=\begin{pmatrix}
 12& -6\sqrt3\\
 -6\sqrt3&9
 \end{pmatrix},
$$
$$
\begin{aligned}
B_{velo}(\psi)
&=
\begin{pmatrix}
(u^2+v^2)^2 &-2(u^2+v^2)uv\\
-2(u^2+v^2)uv&4u^2v^2
\end{pmatrix}
=
\begin{pmatrix}
4&-2\sqrt3\\
-2\sqrt3&3  
\end{pmatrix}.
\end{aligned}
$$
are positive multiples of orthogonal projectors,
with
$$
\text{ ker $A(\psi_*)=\ $ker $B_{ther}(\psi_*)
    =\ $ker $B_{velo}(\psi_*)$ spanned by }
r=\begin{pmatrix}3\\2\sqrt3
  \end{pmatrix}
$$
and
$$
\text{ker }\tilde B_{visc}(\psi_*)\text{ spanned by }
                                     \begin{pmatrix}
                                     3\\ \sqrt3                        
                                     \end{pmatrix}
\not\parallel r. 
$$
As thus both $\eta\tilde B_{visc}(\psi_*)$ and $\mu B_{ther}(\psi_*)+\nu B_{velo}(\psi_*)$  
are nontrivial multiples of orthogonal projectors with different one-dimensional images, 
the matrix $B(\psi_*)$, and so $B(\psi)$ for any $\psi$ near $\psi_*$, is invertible.  
Consequently, the augmented profile equation 
\be
\psi'=(B(\psi))^{-1}F(\psi,\tilde q),\quad \tilde q'=0
\ee
has a 1+1-dimensional center manifold $\mathcal C$ at $(\psi_*,\tilde q_*=\frac34)$.
For any value of $\tilde q$ above and sufficiently close to $\tilde q_*$,  
the 1-dimensional fibre
$$
\mathcal C_{\tilde q}=\{\psi\in\Psi:(\psi,\tilde q)\in\mathcal C\}
$$ 
contains (cf.\ \cite{MP}, p.\ 242) the nearby rest points $\psi_-(\tilde q),\psi_+(\tilde q)$, and as these are the 
only rest points, the segment of the curve $\mathcal C_{\tilde q}$ between them 
is a single orbit $\hat\psi(\R)$ which is heteroclinic to them -- this orbit is 
the sought after traveling wave!  
A straightforward center manifold reduction (cf.\cite{MP}, pp.\ 245, 246) confirms that it has 
$\hat\psi(-\infty)=\psi_-(\tilde q),\hat\psi(+\infty)=\psi_+(\tilde q)$ and not vice versa. 
This finishes the proof of Assertion (i).

Regarding arbitrary shocks, we first show the following properties of $B$. 
\begin{lemma}
Assume $\mu$ and $\nu$ satisfy the causality condition \eqref{generallycausal}. Then\\
(i) 
$\psi_-(\tilde q)$ is an attractor at least for certain 
values of $\tilde q\in (3/4,1)$, unless \eqref{nuisnu*} holds.\\ 
(ii) If \eqref{nuisnu*} holds, $\psi_-(\tilde q)$ is a hyperbolic saddle
for all values of $\tilde q\in (3/4,1)$.
\end{lemma}
\begin{proof}
Relations \eqref{B}--\eqref{Bvelo} readily yield
$$
B(\psi)=
\begin{pmatrix}
 \tilde\eta u^2v^2-16\mu u^2v^2-\nu(u^2+v^2)^2 &
-uv[\tilde\eta u^2-4\mu(4v^2+1)-2\nu(u^2+v^2)]\\
   -uv[\tilde\eta u^2-4\mu(4v^2+1)-2\nu(u^2+v^2)] &
 \tilde\eta u^4-\mu(4v^2+1)^2-4\nu u^2v^2  
\end{pmatrix},
$$
which entails 
\begin{align}
\det B(\psi)&=-9\tilde\eta\mu(1+v^2)v^2-\tilde\eta\nu(1+v^2)^2+\mu\nu(2v^2-1)^2 \nonumber
\\
&=c_4(\tilde\eta,\mu,\nu)v^4+c_2(\tilde\eta,\mu,\nu)v^2+c_0(\tilde\eta,\mu,\nu)
\label{detBofv}
\end{align}
with
$$
\begin{aligned}
c_4(\tilde\eta,\mu,\nu)&=-9\tilde\eta\mu-\tilde\eta\nu+4\mu\nu\\ 
c_2(\tilde\eta,\mu,\nu)&={\color{red}-9\tilde\eta\mu-2\tilde\eta\nu-4\mu\nu}\\
c_0(\tilde\eta,\mu,\nu)&={\color{red}-\tilde\eta\nu+\mu\nu}.
\end{aligned}
$$
Assume first that \eqref{nuisnu*} holds. In this case,
\begin{align}
c_4(\tilde\eta,\mu,\nu)&=0\nonumber\\ 
c_2(\tilde\eta,\mu,\nu)&={\color{red}-\nu(\tilde\eta+8\mu)}\label{c420}\\
c_0(\tilde\eta,\mu,\nu)&={\color{red}\nu(\mu-\tilde\eta)}.
\end{align}
thus
$ 
\det B(\psi^\pm)=\nu(-(\tilde\eta+8\mu)v^2+(\mu-\tilde\eta))<0
$ 
and, as according to \eqref{detA} $\det A(\psi_-)>0$, also 
$$
\det(B(\psi_-)^{-1}A(\psi_-))<0,
$$
which proves (ii).

As $c_4(\tilde\eta,\mu,\nu_*(\eta,\mu))=0$ and 
$\partial c_4(\tilde\eta,\mu,\nu)/\partial\nu=4\mu-\tilde\eta>0$, 
the leading-order coefficient in \eqref{detBofv} is positive 
whenever
\be\label{nubigger}
\nu>\nu_*(\tilde\eta,\mu).
\ee 
Assuming \eqref{nubigger}, we thus have $\det B(\psi_-) >0$ for sufficiently large $v$, which implies that 
\be\label{detB}
\det B(\psi_-)>0\text{ for any $\tilde q<1$ sufficiently close to 1.}
\ee
Since $B(\psi)$ has negative trace and $A(\psi_-)>0$, 
this shows that the eigenvalues of  
$B(\psi_-)^{-1}$ and thus those of $B(\psi_-)^{-1}A(\psi_-) $ are negative and thus 
\be\label{attr}
\psi_-\text{ is an attractor for these same } \tilde q;
\ee
this verifies (i).
\end{proof}

\emph{Proof of Assertion (ii) of Theorem 1}: This is a direct corollary of
Assertion (i) of Lemma 2, as an
attractor cannot be the $\alpha$-limit of any other state; in particular, no heteroclinic 
orbit can originate in $\psi_-$ for any of the $\tilde q$ in \eqref{detB},\eqref{attr}. 
Cf.\ the second plot in Figure 1. \hfill$\Box$

{\bf Remark 1.} Assertion (ii) of Lemma 2 leaves the possibility that in the sharply causal case 
\eqref{sharplycausal}, $\psi_-(\tilde q)$ is connected to $\psi_+(\tilde q)$ by a heteroclinic 
orbit for \emph{all} choices of $\tilde q\in(3/4,1)$.  

{\bf Remark 2.} We conjecture that the phenomena established in this note have counterparts in 
generalizations of the original BDN model to other fluids as have meanhwile been considered in 
\cite{K,BDN20,F20}.   

\appendix
\section{Strict and sharp causality}
For the reader's convenience, we briefly discuss the causality condition \eqref{generallycausal} 
following \cite{BDN}.\footnote{In their notation, \eqref{generallycausal} says that 
cofficients $\chi$ and $\lambda$, which correspond to what here is $3\mu$ and
$\nu$, must satisfy $\chi=a_1\eta, \lambda\ge 3a_1\eta/(a_1-1)$ for some $a_1\ge 4$. 
See \cite{BDN}, p.\ 10, and \cite{F20}.}
Fix $\tilde\eta,\mu,\nu>0$.
The dispersion relation of $B$ being\footnote{We evaluate it in the fluid's rest frame. This means no 
loss of generality as these modes are neutral, cf.\ the proof of Proposition 1 in \cite{FRT}.} 
$$
\begin{aligned}
0&
=\det\left(\lambda^2 B^{\alpha0\gamma0}
+i\xi\lambda(B^{\alpha0\gamma1}+B^{\alpha1\gamma0})
-\xi^2B^{\alpha1\gamma1}\right)
= \left|\begin{matrix}
         9\mu\lambda^2-\nu\xi^2&i\xi\lambda(3\mu+\nu)\\
         i\xi\lambda(3\mu+\nu)&\nu\lambda^2+(\tilde\eta-\mu)\xi^2      
         \end{matrix}
\right|,
\end{aligned}
$$  
the speeds $\sigma=-i\lambda/\xi$ satisfy
$$
0
=\pi(\sigma^2)=9\mu\nu\sigma^4-3\mu(3\tilde\eta+2\nu)\sigma^2
-\nu(\tilde\eta-\mu).
$$
Looking at the signs of the coefficients and the discriminant of the polynomial $\pi$,
$$
\Delta=\tilde\eta\{81\mu^2\tilde\eta+36\mu\nu(3\mu+\nu)\},
$$ 
one sees that the speeds have $\sigma^2\ge 0$ if 
and only if $\mu\ge \tilde\eta$; we assume this now. 

Observing that 
$$
\pi(1)
=4\mu\nu-(9\mu+\nu)\tilde\eta,
$$
one easily concludes that the speeds satisfy $\sigma^2\le 1$ if and only if 
$$
\nu\le\left(\frac1{3\eta}-\frac1{9\mu}\right)^{-1}
$$
holds, with $\sigma^2=1$ occurring exactly in case the last inequality holds with ``=''.

\end{document}